\documentclass[11pt]{article}
\usepackage{color}
\usepackage{epsfig}
\usepackage{subfig}
\usepackage{mathrsfs}
\usepackage{booktabs}
\usepackage{graphicx}
\usepackage{epstopdf}
\usepackage{multirow}
\usepackage{amsmath,amssymb}
\usepackage{amsthm}
\usepackage{fancyhdr}
\usepackage{tabularx}
\usepackage{xfrac}

\newtheorem{coro}{Corollary}[section]
\newtheorem{prop}{Proposition}[section]

\newtheorem{remark}{Remark}[section]

\begin{document}

\title{A Model Predictive Approach to Preventing Cascading Failures of Power Systems}

\author{Chao~Zhai, Hehong~Zhang, Gaoxi Xiao, and~Tso-Chien~Pan
\thanks{Chao Zhai, Hehong Zhang, Gaoxi Xiao and Tso-Chien Pan are with Institute of Catastrophe Risk Management, Nanyang Technological University, 50 Nanyang Avenue, Singapore 639798 and with Future Resilient Systems, Singapore-ETH Centre, 1 Create Way, CREATE Tower, Singapore 138602. Chao Zhai, Hehong Zhang and Gaoxi Xiao are also with School of Electrical and Electronic Engineering, Nanyang Technological University. Corresponding author: Gaoxi Xiao. Email: EGXXiao@ntu.edu.sg}
}

\maketitle

\begin{abstract}
Power system blackouts are usually triggered by the initial contingency and then deteriorate as the branch outage spreads quickly. Thus, it is crucial to eliminate the propagation of cascading outages in its infancy. In this paper, a model predictive approach is proposed to protect power grids against cascading blackout by timely shedding load on buses. The cascading dynamics of power grids is described by a cascading outage model of transmission lines coupled with the DC power flow equation. In addition, a nonlinear convex optimization formulation is established to characterize the optimal load shedding for the mitigation of cascading failures. As a result, two protection schemes are designed on the basis of the optimization formulation. One scheme carries out the remedial action once for all, while the other focuses on the consecutive protection measures. Saddle point dynamics is employed to provide a numerical solution to the proposed optimization problem, and its global convergence is guaranteed in theory. Finally, numerical simulations on IEEE 57 Bus Systems have been implemented to validate the proposed approach in terms of preventing the degradation of power grids.
\end{abstract}


\section{Introduction}\label{sec:int}
The real-time protection of power systems has been a great challenge to researchers for decades since power grids are inevitably exposed to a variety of disruptive disturbances from extreme weather, equipment aging, human errors and even malicious attacks. In fact, it is impossible to achieve the completely reliable and secure operation of power grids due to unexpected contingencies and evolving nature of power grids \cite{beg05}. On the other hand, phasor measurement technology and communication technologies have made great progress in the past decades, which provides powerful tools to detect the real-time state of power systems for emergency control.

The development of power system protection can be divided into three stages in history \cite{bo16}. Specifically, the conventional protection of power systems mainly resorts to electro-mechanical protective relay for severing overloading branches \cite{hew04}.  In spite of high reliability and simplicity in construction, these relays need to be calibrated periodically in the absence of directional feature. The introduction of computers features the second stage of power system protection as advanced intelligent control algorithms can be applied to protect power grids \cite{xyn14}. Furthermore, the utilization of global positioning system (GPS) represents a remarkable milestone, which enables engineers to synchronize time precisely and obtain the global phase information for the wide-area protection \cite{heng14}. The availability of global information on power systems allows to establish a systematic approach to cope with catastrophic scenarios in wide-area power networks. Thus, the special protection scheme (SPS) is proposed to mitigate global stresses by separating power systems into several islands and isolating the faulted areas according to predetermined actions \cite{and96,mad04}. It is demonstrated that the installation of SPSs are economically profitable even with the risk of malfunction \cite{zim02}. Nevertheless, the SPSs are designed for particular power systems suffering from certain abnormal stresses, which inevitably limits their compatibility and the universality for contingencies.

The coordination and control of multi-agent systems has attracted great interests of researchers in various fields including fault identification \cite{cz17,hhz17}, coverage control \cite{cz12} and missile interception \cite{zhai16}. Multi-agent system approaches are also applied to power system control and protection since each bus in power grids can be regarded as an intelligent agent able to exchange electric power and information with its neighboring buses \cite{mca07,tol01,nag02,bab16}. For instance, \cite{mca07} describes fundamental concepts and approaches of multi-agent systems and defines the technical issues in power and energy sector. \cite{tol01} presents a scalable multi-agent paradigm for the control of distributed energy resources in order to achieve higher reliability and more efficient power generation and consumption. \cite{nag02} introduces a multi-agent approach to power system restoration, where two types of agents are defined to achieve the target configuration through mutual interactions. \cite{bab16} proposes an adaptive multi-agent system algorithm to prevent cascading failures without incurring loss. Nevertheless, it has yet to be investigated for the predictive protection of power systems with the approach of multi-agent system.

In this paper, we propose a protection architecture to hinder cascading blackouts by predicting the cascading outages and timely shedding loads on buses. The proposed protection architecture is disturbance-dependent and can be obtained by solving the optimization problem instead of using predetermined preventive actions. Since it has been suggested that the DC power flow is a desirable substitute for the AC power flow in high-voltage transmission networks \cite{yan15,mou09}, for simplicity, the DC power flow equation is employed in this work to compute the power flow and update the connectivity of power networks. The main contributions of our work are listed as follows.
\begin{enumerate}
  \item Propose a disturbance-related real-time protection architecture of power systems by predicting the cascading evolution of transmission lines.
  \item Develop two protection schemes to curb the propagation of cascading outages and succeed in implementing the optimal load shedding with saddle point dynamics.
  \item Compare and verify two protection schemes for more efficient protection and control of power systems during emergency.
\end{enumerate}

The outline of this paper is organized as follows. Section \ref{sec:pre} presents
the protection architecture of power systems against cascading blackouts. Section \ref{sec:ofp} provides the optimization formulation of nonrecurring protection scheme and theoretical analysis, followed by the scheme of recurring protection in Section \ref{sec:cps}. Simulations and validation on IEEE $57$ Bus System are given in Section \ref{sec:sim}.
Finally, we conclude the paper and discuss future work in Section \ref{sec:con}.

\section{Protection Architecture}\label{sec:pre}

\begin{figure}
\scalebox{0.06}[0.06]{\includegraphics{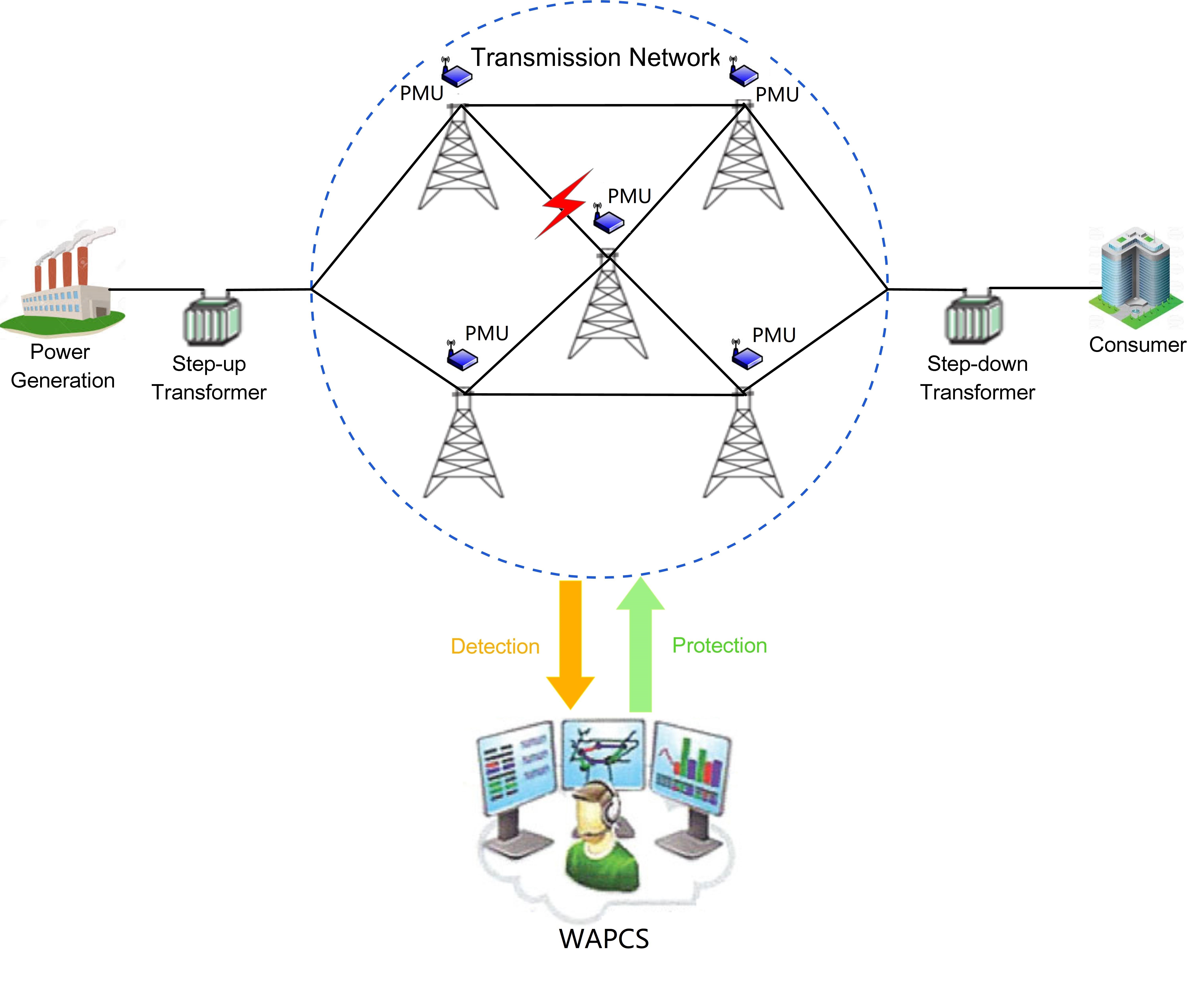}}\centering
\caption{\label{prot} Schematic diagram on monitoring, protection and control of power systems.}
\end{figure}
The operating state of modern power grids is monitored in real time with the aid of phasor measurement units (PMUs) \cite{pha06}, which provide the voltage and current phasor and frequency for wide-area monitoring system (WAMS) \cite{ter11}. This enables the wide-area protection and control system (WAPCS) to identify the disturbances or faults as soon as possible and then take remedial measurements to prevent cascading blackouts. Specifically, the disturbances are detected by PMUs and then transmitted to the phasor data concentrator (PDC) \cite{arm10}. Through the internet, PDCs send the disturbance-related data to data server and wide-area protection and control system, where the cascading process of transmission lines is predicted via the cascading model. Finally, the protection architecture preplans the remedial measures and produces corrective control for the termination of cascading outages (see Fig.~\ref{prot}).

In this work, our goal is to develop a protection architecture that integrates wide-area monitoring, prediction and control of power networks so that WAPCS is able to make disturbance-related remedial actions in order to achieve least power loss during emergency.

First of all, we clarify the concept of cascading step in order to describe the state evolution of power systems during cascading blackout. A cascading step of power systems is defined as one topological change of power networks due to contingencies, human errors or the overloading of transmission lines during cascading process.

To predict the cascading evolution of power systems, it is necessary to obtain the cascading dynamics of power networks, which includes the DC power flow equation and the cascading model of transmission lines. Consider a power network with $n$ buses and $n_b$ branches, and the initial disturbances $\delta$ ($e.g.$, lightening, storm, poor contactor and collapsed vegetation, etc) affect the branch admittance as follows
$$
Y_p^1=Y_p^0+\delta
$$
where $Y_p^0\in R^n$ refers to the $n$-dimensional vector of the {\sl original} branch admittance and $Y_p^1$ denotes the branch admittance at the first cascading step. For simplicity,
the DC power flow equation is employed to compute the power flow on each transmission line.
\begin{equation}\label{dc_power}
P=A^Tdiag(Y_p^k)A\theta^k, \quad k\in\mathbb{N}
\end{equation}
where $P$ denotes the $n_b$-dimensional vector of injected power on each bus and $A\in R^{n\times n_b}$ refers to the branch-bus incidence matrix \cite{stag68}. $\theta^k$ represents the $n_b$-dimensional vector of voltage angle on each bus at the $k$-th cascading step. In addition, the operation $diag(x)$ obtains a square diagonal matrix with the elements of vector $x$ on the main diagonal. The solution to equation (\ref{dc_power}) is expressed as
$$
\theta^k=(A^Tdiag(Y_p^k)A)^{-1^*}P
$$
where the operator $-1^*$ is used to compute the inverse of a square matrix, as defined in \cite{cz17}. According to Lemma 3.2 in \cite{cz17}, the transmission power from Bus $i$ to Bus $j$ at the $k$-th cascading step is given by
\begin{equation}\label{pijk}
P^k_{ij}=e_i^TA^T diag(Y^k_p)Ae_j(e_i-e_j)^T(A^T diag(Y^k_p)A)^{-1^*}P, \quad i,j\in I_{n_b}=\{1,2,...,n_b\}
\end{equation}
where $e_i$ represents the $n_b$ dimensional unit vector with the $i$-th element being equal to $1$ and other elements being equal to $0$. Thus the cascading model of transmission lines is presented as follows
\begin{equation}\label{cas_model}
Y_p^{k+1}=\mathcal{G}(P_{ij}^k)\circ Y_p^k
\end{equation}
where $\circ$ represents the Hadamard product and $\mathcal{G}(P_{ij}^k)$ is given by
$$
\mathcal{G}(P_{ij}^k)=(g(P_{i_1j_1}^k,c_{i_1j_1}),g(P_{i_2j_2}^k,c_{i_2j_2}),...,g(P_{i_nj_n}^k,c_{i_nj_n}))^T \in R^n
$$
with the following approximation function \cite{cz17}
$$
g(P^k_{ij},c_{ij})=\left\{
                \begin{array}{ll}
                  0, & \hbox{$|P^k_{ij}|\geq \sqrt{c_{ij}^2+\frac{\pi}{2\sigma}}$;} \\
                  1, & \hbox{$|P^k_{ij}|\leq\sqrt{c_{ij}^2-\frac{\pi}{2\sigma}}$;} \\
                  \frac{1-\sin\sigma\left[(P^k_{ij})^2-c_{ij}^2\right]}{2}, & \hbox{otherwise.}
                \end{array}
              \right.
$$
where $c_{ij}$ denotes the power threshold of transmission line connecting Bus $i$ to Bus $j$. It is worth pointing out that the approximation function $g(P^k_{ij},c_{ij})$ approaches step function as the tunable parameter $\sigma$ gets close to positive infinity.

\begin{figure}
\scalebox{0.06}[0.06]{\includegraphics{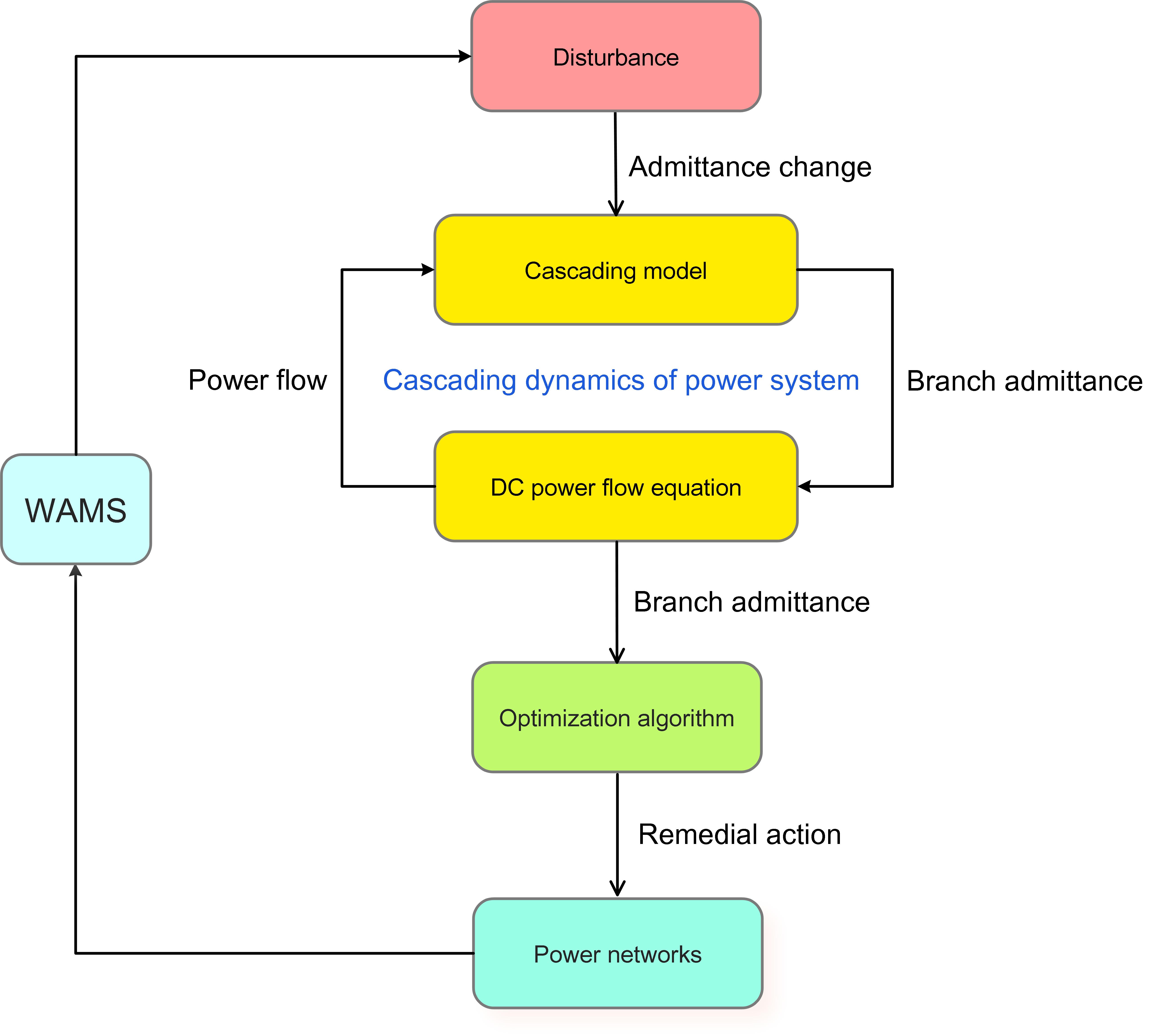}}\centering
\caption{\label{flow} Protection architecture of power systems against cascading failures.}
\end{figure}

Figure \ref{flow} presents the protection architecture of power systems against cascading outages.
On the whole, the protection architecture is composed of three building blocks including disturbance identification, cascading prediction and system protection. The malicious disturbance is detected by the WAMS, then it triggers the state evolution of cascading dynamics, which is a coupling system between the DC power flow equation (\ref{dc_power}) and the cascading model of transmission lines (\ref{cas_model}). Significantly, the cascading dynamics allows us to predict the state evolution of power systems and preplan remedial actions by solving the relevant optimization problem. After power networks are stabilized by the control command from optimization algorithm, the protection architecture starts the detection and identification of disturbances once again. The above process proceeds iteratively to achieve the real-time wide-area protection and control of power grids. Thus, the key task is to design the effective protection scheme and optimization algorithm ($i.e.$, green box in Fig.~\ref{flow}) to terminate the propagation of cascading outages.

In the subsequent two sections, we propose two different protection schemes to prevent the propagation of cascading outages.

\section{Nonrecurring Protection Scheme}\label{sec:ofp}
This section presents the first scheme, $i.e.$ Nonrecurring Protection Scheme (NPS), with which WAPCS implements remedial measures at a given cascading step. Normally, it takes some time for WAMS and WAPCS to identify the disturbance and compute the desired amount of load shedding on each bus when the cascading process occurs. As a result, we need to consider the optimal load shedding of power systems against the propagation of cascading failures at the $m$-th cascading step, and the step number $m$ depends on the propagation speed of cascading outages.  Thus, the optimization problem can be formulated as
\begin{equation}\label{convex}
\begin{split}
    &~~~~~~\min_{P} J(P,W) \\
    &s.t.~~Y^m_p=\mathcal{G}(P_{ij}^{m-1})\circ \mathcal{G}(P_{ij}^{m-2})\circ\cdot\cdot\cdot\circ \mathcal{G}(P_{ij}^1)\circ Y_p^1 \\
    &~~~~~~P^m_{ij}=e_i^TA^T diag(Y^m_p)Ae_j(e_i-e_j)^T(A^T diag(Y^m_p)A)^{-1^*}P \\
    &~~~~~~(P^m_{ij})^2\leq \sigma_{ij}^2, \quad (i,j)\in \mathcal{E} \\
    &~~~~~~\underline{P}_i\leq P_i \leq \bar{P}_i, \quad i\in I_{n_b} \\
\end{split}
\end{equation}
where the objective function $J(P,W)$ is given by
\begin{equation}\label{cost}
J(P,W)=\|W\circ(P-P^0)\|^2
\end{equation}
$P^0$ represents the original power vector injected on each bus, and $P=(P_1,P_2,...P_{n_b})^T$ refers to the injected power vector after load shedding. In addition, the weight vector $W=(W_1,W_2,...,W_{n_b})^T$ characterizes the bus significance in power systems, and $P^m_{ij}$ denotes the transmission power on the branch connecting Bus $i$ and Bus $j$ at the $m$-th cascading step. $\sigma_{ij}$ represents the corresponding power threshold of transmission line.
% and $\epsilon\in [0,1]$ characterizes the tradeoff between load shedding and power consumption. $P^m_e=(P^m_{ij})=diag(Y^m_p)A(A^Tdiag(Y^m_p)A)^{-1^*}P\in R^n$ is the vector of transmission power on branches.

The cost function of Problem (\ref{convex}) characterizes the power mismatch on each bus between the original load distribution and the computed load distribution from the optimization algorithm. The first term in the constraint conditions describes the cascading process of power grids before taking remedial measures. The second term calculates the transmission power on each branch at the $m$-th cascading step, and the third one imposes the restriction on the upper bound of transmission power  on each branch. The final term stipulates the range of load shedding on each bus. Moreover, it is demonstrated that Optimization Problem (\ref{convex}) is convex.

\begin{prop}
Optimization Problem (\ref{convex}) is convex.
\end{prop}

\begin{proof}
The cost function $J(P,W)$ and the constraint function $(P^m_{ij})^2-\sigma^2$ are convex, and $P^m_{ij}$ is affine with respect to $P$. Thus, (\ref{convex}) is a convex optimization problem.
\end{proof}

Then we present the necessary and sufficient condition for the optimal solution to Optimization Problem (\ref{convex}).
\begin{prop}\label{kkt}
Suppose Slater's condition holds (nonempty feasible region) for Convex Optimization Problem (\ref{convex}). Then $P^*$ is the optimal solution if and only if there exist Lagrangian multipliers $\lambda^*_{ij}$, $\bar{\tau}^*_i$ and $\underline{\tau}^*_i$ satisfying the KKT conditions.
$$
\nabla J(P^*,W)+\sum_{i=1}^{n_b}\sum_{j=1}^{n_b}\lambda^*_{ij}\nabla\left[(P^{m*}_{ij})^2-\sigma_{ij}^2\right]
+\sum_{i=1}^{n_b}(\bar{\tau}^*_i-\underline{\tau}^*_i)e_i=\mathbf{0}
$$
and
\begin{equation*}
\begin{split}
    &(P^{m*}_{ij})^2\leq \sigma_{ij}^2 \\
    &P^{m*}_{ij}=e_i^TA^T diag(Y^m_p)Ae_j(e_i-e_j)^T(A^T diag(Y^m_p)A)^{-1^*}P^* \\
    &\underline{P}_i\leq P^*_i \leq \bar{P}_i \\
    & \lambda^*_{ij}[(P^{m*}_{ij})^2-\sigma_{ij}^2]=0 \\
    & \bar{\tau}^*_i(P^*_i-\bar{P}_i)=0 \\
    & \underline{\tau}^*_i(\underline{P}_i-P^*_i)=0
\end{split}
\end{equation*}
\end{prop}

\begin{proof}
The result directly follows from Theorems 3.25-3.27 in \cite{ap06}.
\end{proof}

Design the following Lagrangian function
$$
L(P,\lambda,\tau)=J(P,W)+\sum_{(i,j)\in \mathcal{E}}\lambda_{ij}\left[(P^m_{ij})^2-\sigma_{ij}^2\right]
+\sum_{i=1}^{n_b}\bar{\tau}_i(P_i-\bar{P}_i)+\sum_{i=1}^{n_b}\underline{\tau}_i(\underline{P}_i-P_i)
$$
where $(i,j)$ is an element of set $\mathcal{E}$ if Bus $i$ and Bus $j$ are connected in the network. In addition, we can prove that KKT conditions of convex optimization problem imply a saddle point of the corresponding Lagrangian function.
\begin{prop}\label{iff}
Suppose Slater's condition holds for Optimization Problem (\ref{convex}). The optimal solution $P^*$ to Optimization Problem (\ref{convex}) satisfies the KKT condition in Proposition \ref{kkt} with Lagrangian multipliers $\lambda^*$ and $\tau^*$ if and only if $(P^*,\lambda^*,\tau^*)$ is a
saddle point of the  Lagrangian function $L(P,\lambda,\tau)$.
\end{prop}

\begin{proof}
See Appendix A.
\end{proof}

Next, we present the saddle point dynamics to search for the saddle point of Lagrangian function $L(P,\lambda,\tau)$ \cite{cher17}.
\begin{equation}\label{saddle}
\begin{split}
    \dot{P}&=-\nabla_P~L(P,\lambda,\tau) \\
           &=-2 W\circ(P-P^0) -2\sum_{(i,j)\in\mathcal{E}}\lambda_{ij}P^m_{ij}R^m_{ij}-(\bar{\tau}-\underline{\tau}) \\
    \dot{\lambda}_{ij}&=[\nabla_{\lambda_{ij}}~L(P,\lambda,\tau)]^+_{\lambda_{ij}}=[(P^m_{ij})^2-\sigma_{ij}^2]^+_{\lambda_{ij}} \\
    \dot{\bar{\tau}}_i&=[\nabla_{\bar{\tau}_i}~L(P,\lambda,\tau)]^+_{\bar{\tau}_i}=[P_i-\bar{P}_i]^+_{\bar{\tau}_i}\\
    \dot{\underline{\tau}}_i&=[\nabla_{\underline{\tau}_i}~L(P,\lambda,\tau)]^+_{\underline{\tau}_i}=[\underline{P}_i-P_i]^+_{\underline{\tau}_i}
\end{split}
\end{equation}
where
$$
R^m_{ij}=e_i^TA^T diag(Y^m_p)Ae_j\cdot\left[(A^T diag(Y^m_p)A)^{-1^*}\right]^T(e_i-e_j)
$$
and the operator $[~]^+$ is defined as
\begin{equation}\label{switch}
[x]^+_y=\left\{
          \begin{array}{ll}
            x, & \hbox{$y>0$;} \\
            \max\{x,0\}, & \hbox{$y=0$.}
          \end{array}
        \right.
\end{equation}
%{\color{blue}
%\begin{remark}
%Saddle Point Dynamics (\ref{saddle}) can be implemented in a distributed manner because the term $(e_i-e_j)^T(A^T diag(Y_p)A)^{-1^*}P$ can be solved with a distributed algorithm \cite{mou15}.
%\end{remark}
%}
Significantly, it is guaranteed in theory that Saddle Point Dynamics (\ref{saddle}) approaches the optimal solution of Problem (\ref{convex}) as time goes to infinity.
\begin{prop}\label{converge}
Saddle Point Dynamics (\ref{saddle}) globally asymptotically converges to the optimal solution to Optimization Problem (\ref{convex}).
\end{prop}

\begin{proof}
See Appendix B.
\end{proof}

\begin{table}
 \caption{\label{MPA1} Nonrecurring Protection Scheme.}
 \begin{center}
 \begin{tabular}{lcl} \hline
  1: Set the running time $T$ and the cascading step $m$ \\
  2: \textbf{while} ($t<=T$) \\
  3: ~~~~~~~Detect network topology at time $t$ \\
  4: ~~~~~~~Update the branch-bus incidence matrix $A(t)$  \\
  5: ~~~~~~~\textbf{if} ($A(t)\neq A$)  \\
  6: ~~~~~~~~~~~Identify the disturbance $\delta=Y_p^1-Y_p^0$ \\
  7: ~~~~~~~~~~~Initiate cascading dynamics with (\ref{pijk}) and (\ref{cas_model})  \\
  8: ~~~~~~~~~~~Save system state at the $m$-th cascading step \\
  9: ~~~~~~~~~~~Solve Saddle Point Dynamics (\ref{saddle}) \\
 10: ~~~~~~~~ Implement load shedding based on the solution to (\ref{saddle}) \\
 11: ~~~~~~~\textbf{end if}  \\
 12: ~~~~~~~Compute the power flow on each branch \\
 13: \textbf{end while} \\ \hline
 \end{tabular}
 \end{center}
\end{table}

Finally, we present the Nonrecurring Protection Scheme in Table \ref{MPA1}. First of all, we set the running time of NPS and the number of cascading steps $m$. In reality, the running time $T$ is sufficient large since the protection scheme is recurrent. Then NPS detects the network topology in real time with the aid of WAMS and locates the targeted branches by comparing the current branch-bus incidence matrix $A(t)$ with the original one $A$. After that, NPS starts the computation of admittance changes on the targeted branches to identify the disturbance. The above disturbance initiates the cascading process to predict the state of power grids at the $m$-th cascading step, which enables us to solve Saddle Point Dynamics (\ref{saddle}) and produce the remedial actions of optimal load shedding to prevent cascading outages at the $m$-th cascading step. Finally, NPS computes the power flow on each branch to validate the remedial measurements.

\section{Recurring Protection Scheme}\label{sec:cps}
In this section, we propose the second scheme, $i.e.$ Recurring Protection Scheme (RPS), with which WAPCS takes remedial actions and implements corrective control by shedding load at the consecutive cascading steps. Compared with NPS, more control variables are available to optimize the objective function in the current formulation. Actually, NPS can be regarded as a special case of RPS by stipulating the load shedding at a given cascading step. Essentially, RPS increases the flexibility with respect to terminating cascading outages. Thus, RPS appears to be superior to the first one in terms of mitigating the power loss. Theoretically, the optimization problem can be formulated as
\begin{equation}\label{cascon}
\begin{split}
    &~~~~~~\min_{P^k}~\mathcal{C}(P^1,P^2,...,P^m,W) \\
    &s.t.~~P^k_{ij}=e_i^TA^T diag(Y^k_p)Ae_j(e_i-e_j)^T(A^T diag(Y^k_p)A)^{-1^*}P^k \\
    &~~~~~~Y^{k+1}_p=\mathcal{G}(P_{ij}^{k})\circ Y_p^{k} \\
    &~~~~~~(P^m_{ij})^2\leq \sigma_{ij}^2, \quad (i,j)\in \mathcal{E} \\
    &~~~~~~\underline{P}_i\leq P^k_i \leq \bar{P}_i, \quad i\in I_{n_b}, \quad k\in I_m \\
\end{split}
\end{equation}
where the objective function is given by
\begin{equation}\label{cost_m}
    \mathcal{C}(P^1,P^2,...,P^m,W)=\sum_{k=1}^{m}\|W\circ(P^k-P^0)\|^2
\end{equation}
and $W$ represents the weight vector.
\begin{prop}\label{outper}
Solutions to Optimization Problem (\ref{cascon}) outperform those to Convex Optimization Problem (\ref{convex}) in terms of minimizing objective function.
\end{prop}

\begin{proof}
Let $P^{k*},k\in I_m$ denote the solution to Optimization Problem (\ref{cascon}), and $P^*$ represents the solution to Convex Optimization Problem (\ref{convex}). Then it follows from
$$
\|W\circ(P^{m*}-P^0)\|^2\leq \mathcal{C}(P^{1*},P^{2*},...,P^{m*},W)\leq \mathcal{C}(P^0,P^0,...,P^0,P^*,W)=\|W\circ(P^*-P^0)\|^2
$$
that
$$
\|W\circ(P^{m*}-P^0)\|^2\leq\|W\circ(P^*-P^0)\|^2,
$$
which completes the proof.
\end{proof}

In practice, it is impractical to implement the load shedding at each cascading step since WAPCS has to take some time in order to work out the optimization algorithm. Thus, we consider the remedial actions on two consecutive cascading steps ($i.e.$, $m-1$ and $m$) and reformulate Optimization Problem (\ref{cascon}) as
\begin{equation}\label{cascon1}
\begin{split}
    &~~~~~~\min_{P^k}~\mathcal{C}(P^{m-1},P^m,W) \\
    &s.t.~~P^k_{ij}=e_i^TA^T diag(Y^k_p)Ae_j(e_i-e_j)^T(A^T diag(Y^k_p)A)^{-1^*}\tilde{P}^k \\
    &~~~~~~Y^{k+1}_p=\mathcal{G}(P_{ij}^{k})\circ Y_p^{k} \\
    &~~~~~~(P^m_{ij})^2\leq \sigma_{ij}^2, \quad (i,j)\in \mathcal{E} \\
    &~~~~~~\underline{P}_i\leq \tilde{P}^k_i \leq \bar{P}_i, \quad i\in I_{n_b}, \quad k\in I_m \\
\end{split}
\end{equation}
where
$$
\mathcal{C}(P^{m-1},P^m,W)=\|W\circ(P^{m-1}-P^0)\|^2+\|W\circ(P^m-P^0)\|^2
$$
and
$$
\tilde{P}^k=\left\{
              \begin{array}{ll}
                P^0, & \hbox{$k\leq m-2$;} \\
                P^k, & \hbox{otherwise.}
              \end{array}
            \right.
$$

\begin{coro}
Solutions to Optimization Problem (\ref{cascon1}) outperform those to Convex Optimization Problem (\ref{convex}) in terms of minimizing objective function.
\end{coro}
\begin{proof}
The proof follows by the same argument as that in Proposition \ref{outper}, and is thus omitted.
\end{proof}

Actually, it is difficult to obtain the optimal solution to Optimization Problem (\ref{cascon1}) due to its non-convexity. The linearized method can be applied to approximate the non-convex constraint for achieving the global optima. To simplify mathematical expressions, we define
$$
\mathcal{F}(Y_p^k)=e_i^TA^T diag(Y^k_p)Ae_j(e_i-e_j)^T(A^T diag(Y^k_p)A)^{-1^*}
$$
Then we obtain $P^k_{ij}=\mathcal{F}(Y_p^k)\tilde{P}^k$ and
\begin{equation*}
\begin{split}
P^m_{ij}&=\mathcal{F}(Y_p^m)P^m\\
        &=\mathcal{F}\left[\mathcal{G}(P_{ij}^{m-1})\circ Y_p^{m-1}\right]P^m \\
        &=\mathcal{F}\left[\mathcal{G}(\mathcal{F}(Y_p^{m-1})P^{m-1})\circ Y_p^{m-1}\right]P^m
\end{split}
\end{equation*}
The gradient of $P^m_{ij}$ with respect to $[P^{m-1},P^{m}]$ is given by
\begin{equation*}
\begin{split}
\nabla P^m_{ij}&=\left(
                   \begin{array}{c}
                     \frac{\partial}{\partial P^{m-1}}\mathcal{F}\left[\mathcal{G}(\mathcal{F}(Y_p^{m-1})P^{m-1})\circ Y_p^{m-1}\right]P^m\\
                     \frac{\partial}{\partial P^m}\mathcal{F}\left[\mathcal{G}(\mathcal{F}(Y_p^{m-1})P^{m-1})\circ Y_p^{m-1}\right]P^m \\
                   \end{array}
                 \right)\\
               &=\left(
                   \begin{array}{c}
                     \nabla_{Y_p^m}\left[\mathcal{F}(Y_p^m)P^m\right]^T\left[\mathcal{G}'(\mathcal{F}(Y_p^{m-1})P^{m-1})\circ Y_p^{m-1}\right]\cdot\mathcal{F}(Y_p^{m-1})^T\\
                     \mathcal{F}\left[\mathcal{G}(\mathcal{F}(Y_p^{m-1})P^{m-1})\circ Y_p^{m-1}\right]^T \\
                   \end{array}
                 \right)
\end{split}
\end{equation*}
Therefore, $P^m_{ij}=\mathcal{F}(Y_p^m)P^m$ can be approximated by the following linear equation in the neighborhood of variables $[P^0,P^0]$.
\begin{equation}\label{pijm}
\begin{split}
\hat{P}^m_{ij}|_{[P^0,P^0]}&=\mathcal{F}\left[\mathcal{G}(\mathcal{F}(Y_p^{m-1})P^{0})\circ Y_p^{m-1}\right]P^0+\left(\nabla P^m_{ij}|_{[P^0,P^0]}\right)^T\left(
                                       \begin{array}{c}
                                         P^{m-1}-P^0 \\
                                         P^m-P^0\\
                                       \end{array}
                                     \right)\\
        &=\mathcal{F}\left[\mathcal{G}(\mathcal{F}(Y_p^{m-1})P^{0})\circ Y_p^{m-1}\right]P^0\\
        &+\nabla_{Y_p^m}\left[\mathcal{F}(Y_p^m|_{P^0})P^0\right]^T\left[\mathcal{G}'(\mathcal{F}(Y_p^{m-1})P^{0})\circ Y_p^{m-1}\right]\cdot\mathcal{F}(Y_p^{m-1})(P^{m-1}-P^0)\\
        &+\mathcal{F}\left[\mathcal{G}(\mathcal{F}(Y_p^{m-1})P^{0})\circ Y_p^{m-1}\right](P^m-P^0)\\
        &=\nabla_{Y_p^m}\left[\mathcal{F}(Y_p^m|_{P^0})P^0\right]^T\left[\mathcal{G}'(\mathcal{F}(Y_p^{m-1})P^{0})\circ Y_p^{m-1}\right]\cdot\mathcal{F}(Y_p^{m-1})(P^{m-1}-P^0)\\
        &+\mathcal{F}\left[\mathcal{G}(\mathcal{F}(Y_p^{m-1})P^{0})\circ Y_p^{m-1}\right]P^m\\
\end{split}
\end{equation}
where $Y_p^m|_{P^0}=\mathcal{G}(\mathcal{F}(Y_p^{m-1})P^0)\circ Y_p^{m-1}$.
In this way, Optimization Problem (\ref{cascon1}) can be approximated by the following problem.
\begin{equation}\label{casconv}
\begin{split}
    &~~~~~~\min_{P^k}~\mathcal{C}(P^{m-1},P^m,W) \\
    &s.t.~~\hat{P}^m_{ij}=\nabla_{Y_p^m}\left[\mathcal{F}(Y_p^m|_{P^0})P^0\right]^T\left[\mathcal{G}'(\mathcal{F}(Y_p^{m-1})P^{0})\circ Y_p^{m-1}\right]\cdot\mathcal{F}(Y_p^{m-1})(P^{m-1}-P^0)\\
    &~~~~~~~~~~~+\mathcal{F}\left[\mathcal{G}(\mathcal{F}(Y_p^{m-1})P^{0})\circ Y_p^{m-1}\right]P^m\\
    &~~~~~~(\hat{P}^m_{ij})^2\leq \sigma_{ij}^2, \quad (i,j)\in \mathcal{E} \\
    &~~~~~~\underline{P}_i\leq P^k_i \leq \bar{P}_i, \quad i\in I_{n_b}, \quad k\in \{m-1,m\} \\
\end{split}
\end{equation}

\begin{prop}
Optimization Problem (\ref{casconv}) is convex.
\end{prop}

\begin{proof}
The objective function $\mathcal{C}(P^{m-1},P^m,W)$ is convex, and $P^m_{ij}$ is an affine function of variables $P^{m-1}$ and $P^m$. Moreover, $(\hat{P}^m_{ij})^2-\sigma^2$ is convex as well. This indicates that Problem (\ref{casconv}) is convex.
\end{proof}

Next, we discuss the numerical solution to Optimization Problem (\ref{casconv}). Design the Lagrangian function for Problem (\ref{casconv}) as follows
\begin{equation*}
\begin{split}
L(P^{m},P^{m-1},\lambda,\tau)&=\mathcal{C}(P^{m-1},P^m,W)\\
&+\sum_{i=1}^{n_b}\left[\bar{\tau}^m_i(P^m_i-\bar{P}_i)+\bar{\tau}^{m-1}_i(P^{m-1}_i-\bar{P}_i)\right]\\
&+\sum_{i=1}^{n_b}\left[\underline{\tau}^m_i(\underline{P}_i-P^m_i)+\underline{\tau}^{m-1}_i(\underline{P}_i-P^{m-1}_i)\right]\\
&+\sum_{(i,j)\in \mathcal{E}}\lambda_{ij}\left[(\hat{P}^m_{ij})^2-\sigma_{ij}^2\right]
\end{split}
\end{equation*}

Saddle point dynamics to search for the saddle point of Lagrangian function $L(P^m,P^{m-1},\lambda,\tau)$ is given by
\begin{equation}\label{saddlecas}
\begin{split}
    \dot{P^k}&=-\nabla_{P^k}~L(P^m,P^{m-1}\lambda,\tau) \\
             &=-2W\circ (P^k-P^0)-2\sum_{(i,j)\in\mathcal{E}}\lambda_{ij}\hat{P}^m_{ij}R^k_{ij}-(\bar{\tau}^k-\underline{\tau}^k) \\
    \dot{\lambda}_{ij}&=[\nabla_{\lambda_{ij}}~L(P,\lambda,\tau)]^+_{\lambda_{ij}}=[(\hat{P}^m_{ij})^2-\sigma_{ij}^2]^+_{\lambda_{ij}} \\
    \dot{\bar{\tau}}^k_i&=[\nabla_{\bar{\tau}^k_i}~L(P,\lambda,\tau)]^+_{\bar{\tau}^k_i}=[P^k_i-\bar{P}_i]^+_{\bar{\tau}^k_i}\\
    \dot{\underline{\tau}}^k_i&=[\nabla_{\underline{\tau}^k_i}~L(P,\lambda,\tau)]^+_{\underline{\tau}^k_i}=[\underline{P}_i-P^k_i]^+_{\underline{\tau}^k_i}
\end{split}
\end{equation}
where $k\in\{m-1,m\}$ and
$$
R^k_{ij}=\left\{
           \begin{array}{ll}
             \nabla_{Y_p^m}\left[\mathcal{F}(Y_p^m|_{P^0})P^0\right]^T\left[\mathcal{G}'(\mathcal{F}(Y_p^{m-1})P^{0})\circ Y_p^{m-1}\right]\cdot\mathcal{F}(Y_p^{m-1})^T, & \hbox{$k=m-1$;} \\
             \mathcal{F}\left[\mathcal{G}(\mathcal{F}(Y_p^{m-1})P^{0})\circ Y_p^{m-1}\right]^T, & \hbox{$k=m$.}
           \end{array}
         \right.
$$
and the operator $[~]^+$ is defined in equation (\ref{switch}).

\begin{remark}
The solution to optimization problem (\ref{casconv}) merely guarantees $\hat{P}^m_{ij}\leq \sigma_{ij}$ instead of $P^m_{ij}\leq \sigma_{ij}, \forall (i,j)\in\mathcal{E}$. Therefore, it is necessary to check whether the constraints $P^m_{ij}\leq \sigma_{ij}, \forall (i,j)\in\mathcal{E}$ hold after shedding load at the $(m-1)$-th step and the $m$-th step according to the solution to (\ref{casconv}). The solution to (\ref{saddle}) can be adopted as a remedy if there exists $(i,j)\in\mathcal{E}$ such that $P^m_{ij}>\sigma_{ij}$.
\end{remark}

In theory, we can guarantee that Saddle Point Dynamics (\ref{saddlecas}) achieves the asymptotic convergence of global optimal solution to Optimization Problem (\ref{casconv}).
\begin{prop}\label{cp_con}
Saddle Point Dynamics (\ref{saddlecas}) globally asymptotically converges to the optimal solution to Convex Optimization Problem (\ref{casconv}).
\end{prop}
\begin{proof}
See Appendix C.
\end{proof}

\begin{table}
 \caption{\label{MPA2} Recurring Protection Scheme.}
 \begin{center}
 \begin{tabular}{lcl} \hline
  1: Set the running time $T$ and the cascading step $m$ \\
  2: \textbf{while} ($t<=T$) \\
  3: ~~~~~~~Detect network topology at time $t$ \\
  4: ~~~~~~~Update the branch-bus incidence matrix $A(t)$  \\
  5: ~~~~~~~\textbf{if} ($A(t)\neq A$)  \\
  6: ~~~~~~~~~~~Identify the disturbance $\delta=Y_p^1-Y_p^0$ \\
  7: ~~~~~~~~~~~Initiate cascading dynamics with (\ref{pijk}) and (\ref{cas_model})  \\
  8: ~~~~~~~~~~~Get system state at the $\left(m-1\right)$-th cascading step \\
  9: ~~~~~~~~~~~Solve Saddle Point Dynamics (\ref{saddlecas}) \\
  10: ~~~~~~~~~~Calculate $P^m_{ij}, \forall(i,j)\in \mathcal{E}$ with the solution to (\ref{saddlecas}) \\
  11: ~~~~~~~~~~\textbf{if} ($|P^m_{ij}|\leq\sigma_{ij},\forall(i,j)\in \mathcal{E}$)\\
  12: ~~~~~~~~~~~~~~Shed load according to the solution to (\ref{saddlecas}) \\
  13: ~~~~~~~~~~\textbf{else} \\
  14: ~~~~~~~~~~~~~~Shed load according to the solution to (\ref{saddle}) \\
  15: ~~~~~~~~~~\textbf{end if} \\
  16: ~~~~~~~~~~Compute the power flow on each branch \\
  17: ~~~~~~~\textbf{end if}  \\
  18: \textbf{end while} \\ \hline
 \end{tabular}
 \end{center}
\end{table}

Table \ref{MPA2} presents the procedure of RPS, which resembles NPS except for controllable cascading steps and saddle point dynamics. Compared with NPS, RPS imposes the operation of optimal load shedding in two consecutive cascading steps to terminate cascading outages. If RPS fails to prevent further cascading outages through the model validation of two consecutive load shedding, the one-off load shedding from the solution to (\ref{saddle}) will take effect as a remedy.

\section{Numerical Simulations}\label{sec:sim}

\begin{figure}
\scalebox{0.55}[0.55]{\includegraphics{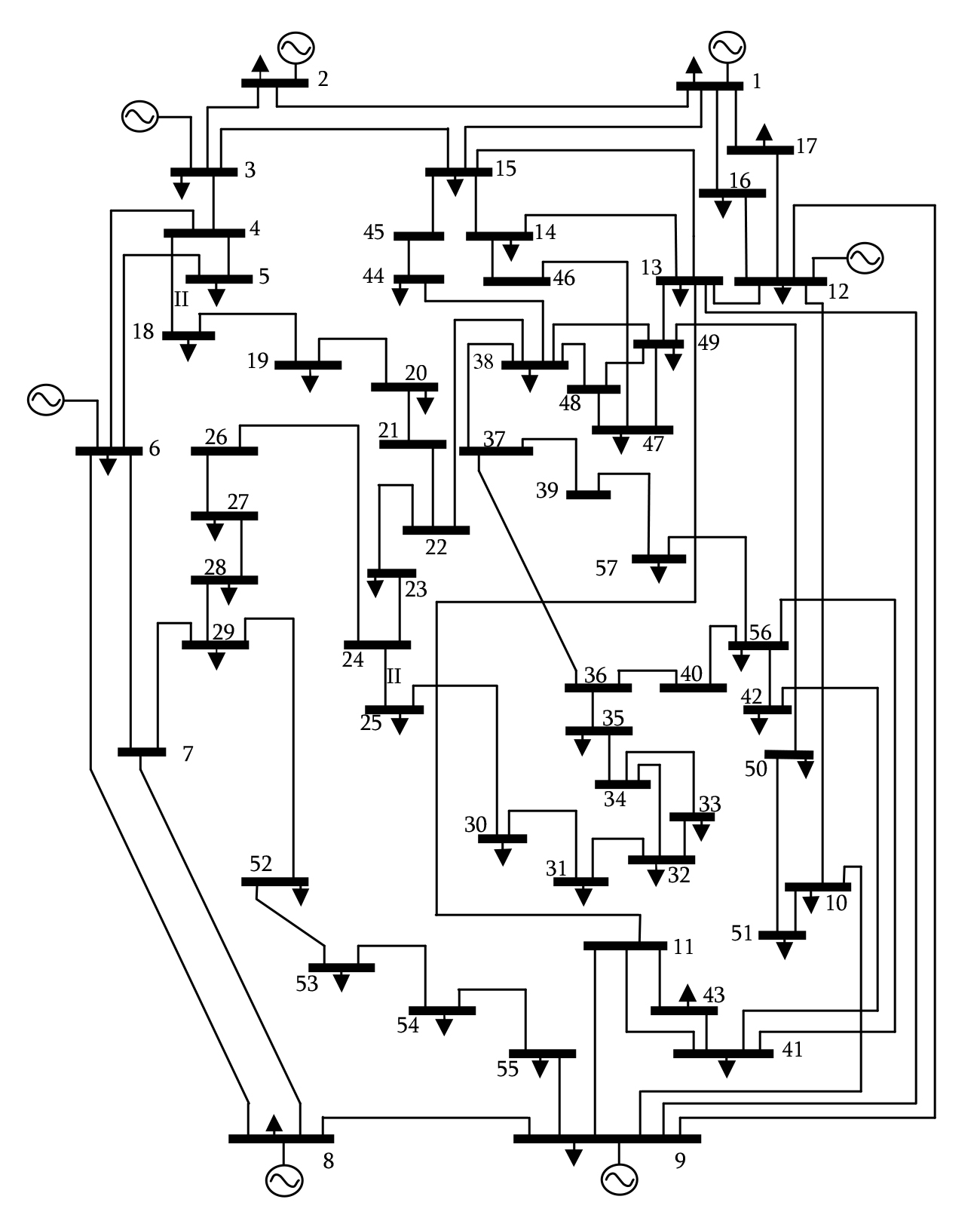}}\centering
\caption{\label{bus57} IEEE 57 Bus System.}
\end{figure}

In this section, we implement and compare two protection schemes ($i.e.$, NPS and RPS) on IEEE 57 Bus Systems (see Fig.~\ref{bus57}).

\subsection{Nonrecurring Protection Scheme}
The parameter setting for NPS in Table 1 is given in detail as follows. For simplicity, we set $W=(1,1,...,1)^T$ to ensure that all buses are equally treated in load shedding.
Per unit values are adopted with the base value of power 100MV in the simulations.
The minimum injected power on each bus is equal to the negative value of its total load, and the maximum injected power is provided by the generator connected to this bus. For load buses, their maximum injected power is $0$. Consequently, for load buses, we set $\bar{P}_i=0$ and $\underline{P}_i$ is the total injected power. For generator buses, $\bar{P}_i$ is the total power from the generator, while $\underline{P}_i$ is the injected power from the load. The threshold of power transmission on each branch is $1$ pu with $\sigma=10^3$ in the approximation function, and the operation of optimal load shedding is implemented at the $4$th cascading step ($i.e.$, $m=4$) according to the solution to Optimization Problem (\ref{convex}). In addition, Euler method is employed to implement the numerical algorithm of saddle point dynamics (\ref{saddle}) with time interval $0.1$s.
%The threshold of transmission power on each branch is given as $\sigma_{ij}=1.5$ pu, $\forall (i,j)\in \mathcal{E}$.

Branch $10$ is severed as the initial disturbance of power systems, and Figure~\ref{nop} shows the cascading process of the IEEE 57 Bus Systems without protection schemes. After 6 cascading steps, the system ends up with $43$ connected branches ($5$ branches with the transmission power) and the total transmission power of $1.004$pu. In contrast, Figure~\ref{a1cas} presents the topology evolution of the IEEE 57 Bus Systems with NPS at the $4$th cascading step. It is observed that the power system is well protected by NPS since most branches in the network are in a good state of transmitting power among buses. Specifically, the cascading process stops after implementing the optimal load shedding at the $4$th cascading step, and the system remains unchanged with the total transmission power of $9.156$pu and 53 connected branches in operation and no idle branches. Moreover, the objective function is minimized with the value of 0.1068. The upper panel in Fig.~\ref{mpa1} describes the evolution of Saddle Point Dynamic (\ref{saddle}) within $10$s, and the trajectories get stable after $4$s. The lower panel shows the load shedding on each bus at the $4$th cascading step. There are no negative amounts of load shedding, which implies the absence of generator tripping during power systems protection.

\begin{figure}
\scalebox{0.8}[0.95]{\includegraphics{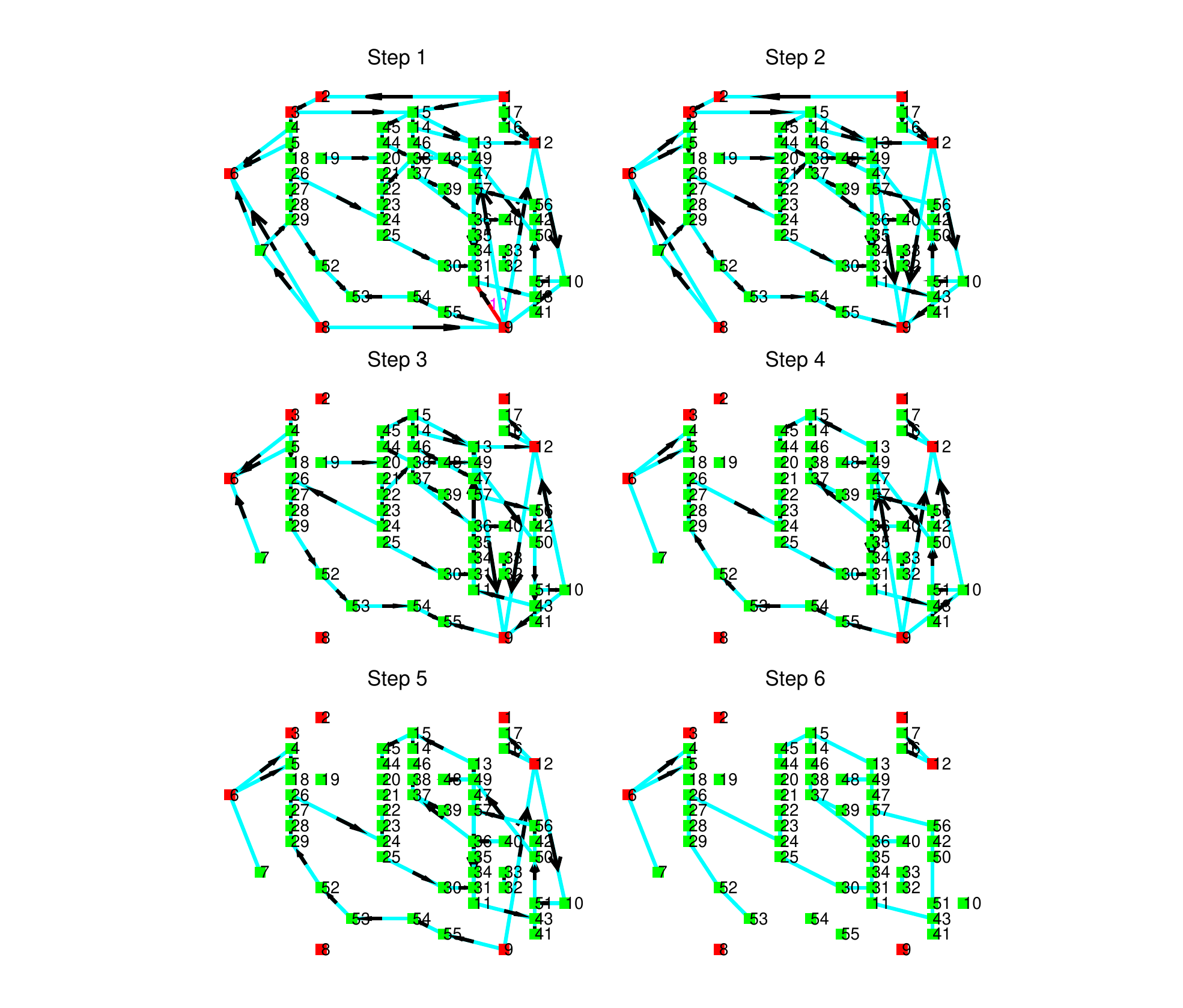}}\centering
\caption{\label{nop} Cascading process of the IEEE 57 Bus System without protection. Red squares represent the generator buses, and green ones denote the load buses. The arrow on each branch refers to the power flow, and it disappears if there is no transmission power on the branch. The initial disturbance is added on Branch $10$ (red link) by severing this branch.}
\end{figure}

\begin{figure}
\scalebox{0.8}[0.8]{\includegraphics{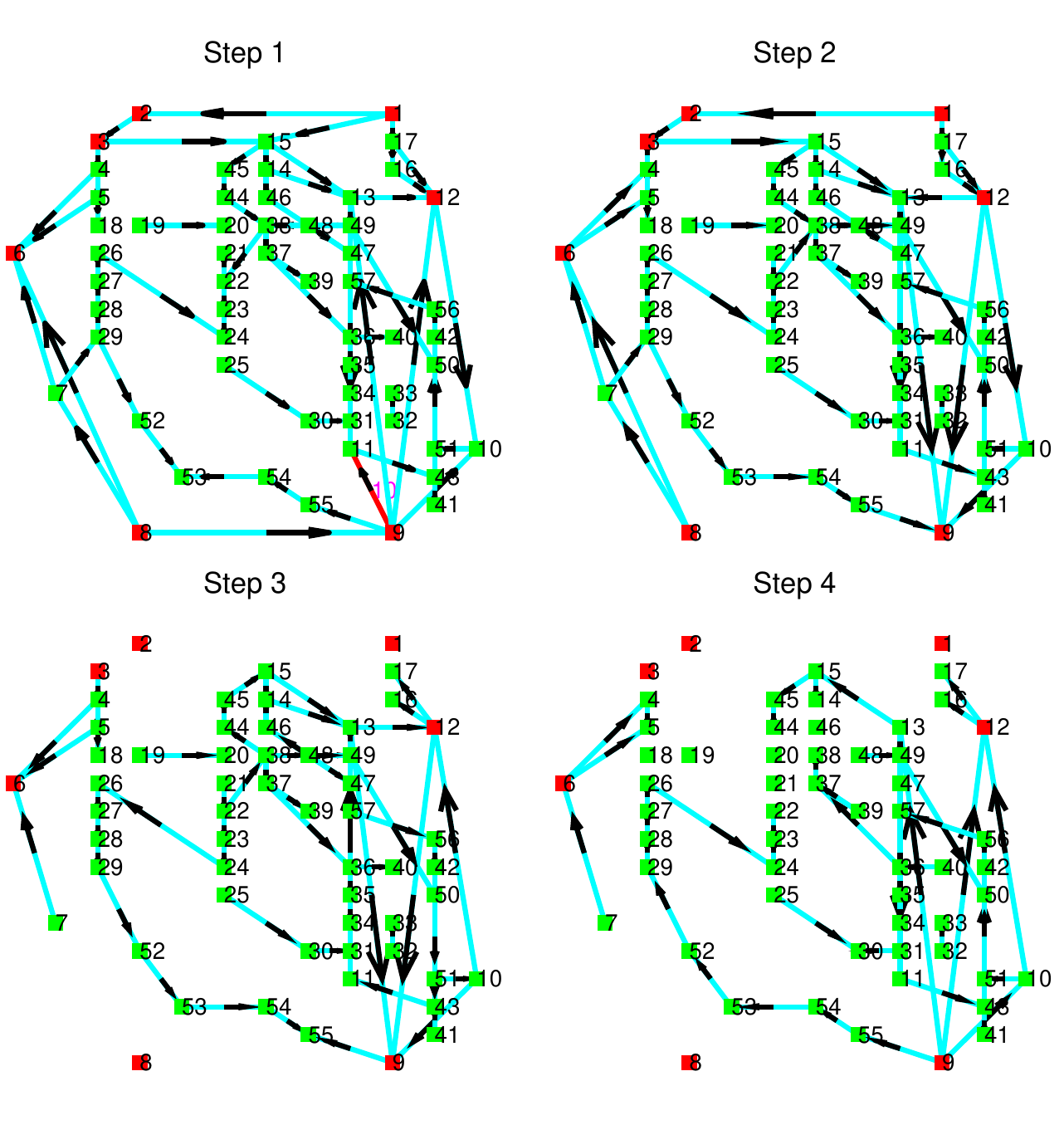}}\centering
\caption{\label{a1cas} Cascading process of the IEEE 57 Bus System protected by NPS.}
\end{figure}

\begin{figure}\centering
\scalebox{0.6}[0.6]{\includegraphics{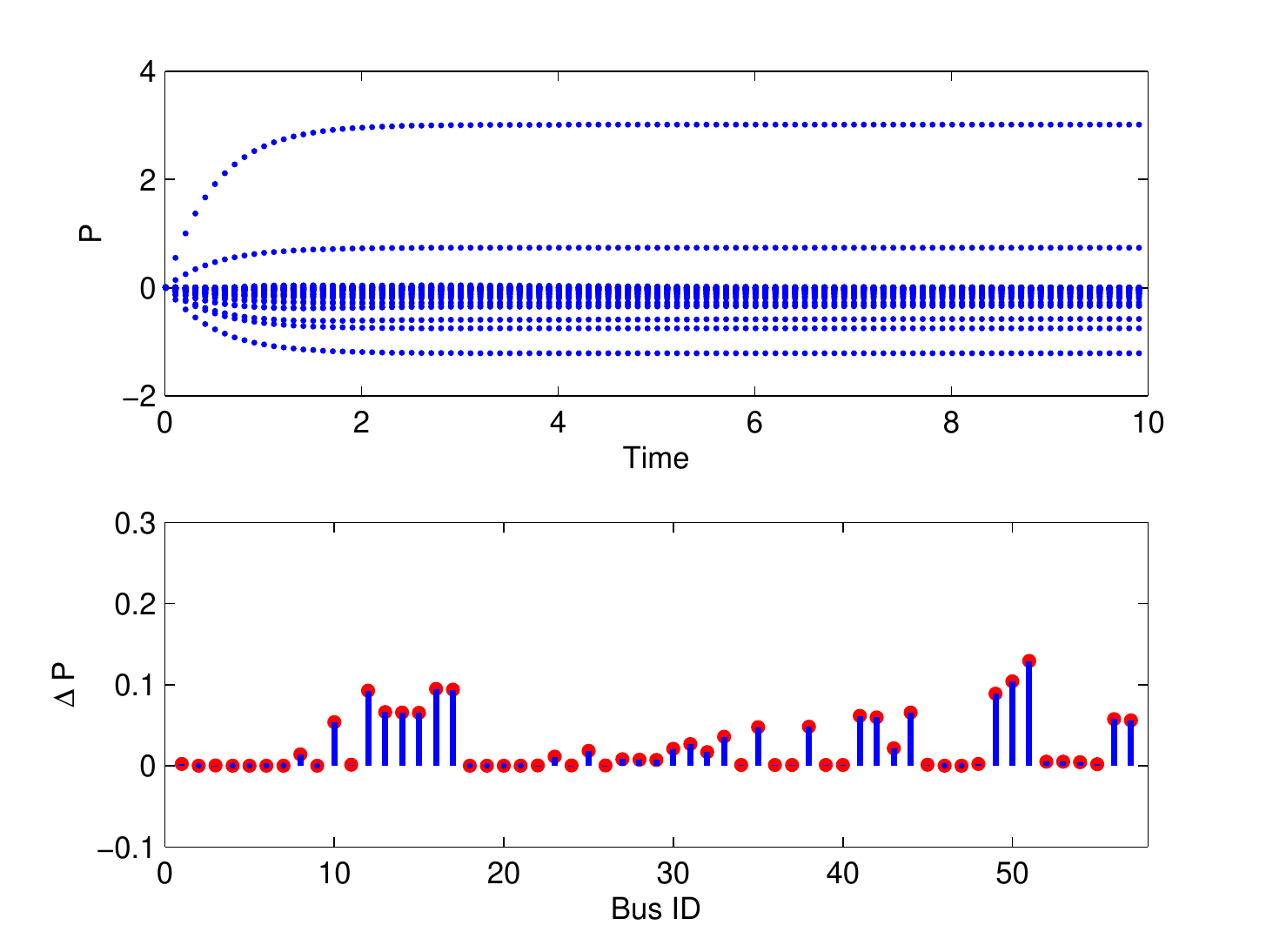}}\centering
 \caption{\label{mpa1} Time evolution of saddle point dynamics and load shedding in NPS.}
\end{figure}

\subsection{Recurring Protection Scheme}
The parameter setting of RPS is the same as that of NPS except that the operation of load shedding is implemented at both the $3$rd and the $4$th cascading steps. Figure~\ref{a2cas} presents the cascading process of the power system with RPS. After shedding load according to the solution of Saddle Point Dynamics (\ref{saddlecas}) at two consecutive cascading steps, the cascading process terminates at the $4$th cascading step. Figure~\ref{mpa2} presents the state trajectories of Saddle Point Dynamics (\ref{saddlecas}) within $10$s and the distribution of load shedding on each bus according to RPS. Moreover, the total transmission power is $12.496$pu with 53 connected branches and no idle branches. The value of objective function is $0.0979$ in the end, less than $0.1068$ in NPS. Additionally, the power mismatch at the $4$th cascading step is $0.0732$, which indicates the superiority of RPS. As observed in the upper panel, the trajectories of Saddle Point Dynamics (\ref{saddlecas}) converge to stable values after $4$s. In the lower panel, the red line segments denote the amount of load shedding at the $3$rd cascading step, while the green ones refer to those at the $4$th cascading step. It is worth pointing out that the negative value at Bus 12 indicates the decrease of power supply to this generator bus. To make a comprehensive analysis, we also carry out the load shedding of RPS at the $4$th and $5$th cascading steps. It is observed that the cascading outages come to an end with 53 connected branches in operation and the total transmission power of $10.799$pu. The objective function is minimized to $0.0919$, which is still less than $0.1068$ in NPS. This demonstrates the better performance of RPS to maintain the power transmission and restrain cascading outages in spite of the high computational complexity.

\begin{figure}
\scalebox{0.8}[0.8]{\includegraphics{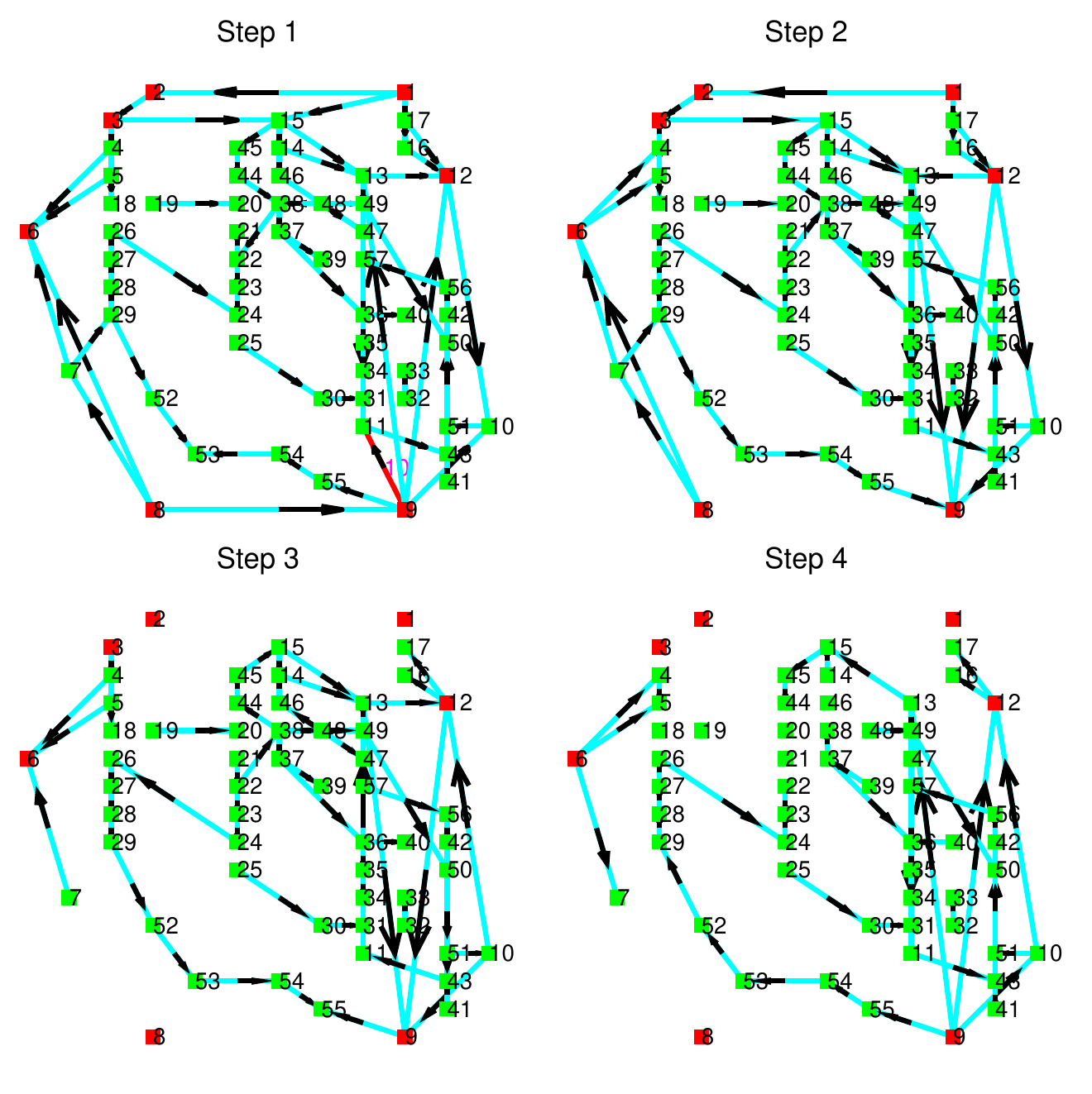}}\centering
\caption{\label{a2cas} Cascading process of the IEEE 57 Bus System protected by RPS.}
\end{figure}

\begin{figure}\centering
\scalebox{0.6}[0.6]{\includegraphics{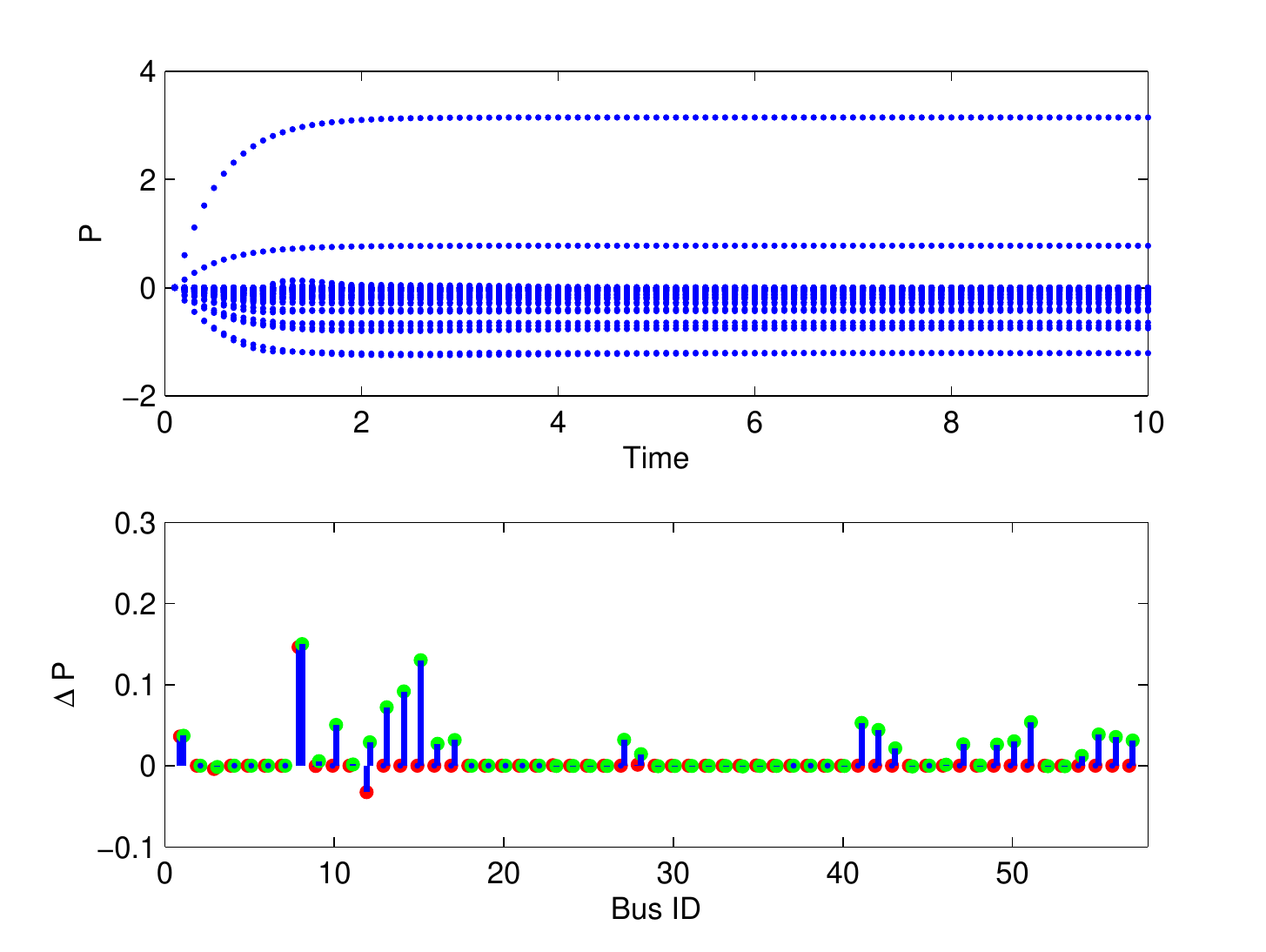}}\centering
 \caption{\label{mpa2} Time evolution of saddle point dynamics and load shedding in RPS.}
\end{figure}

\subsection{The Effect of Parameter $W$}
\begin{figure}
\scalebox{0.8}[0.8]{\includegraphics{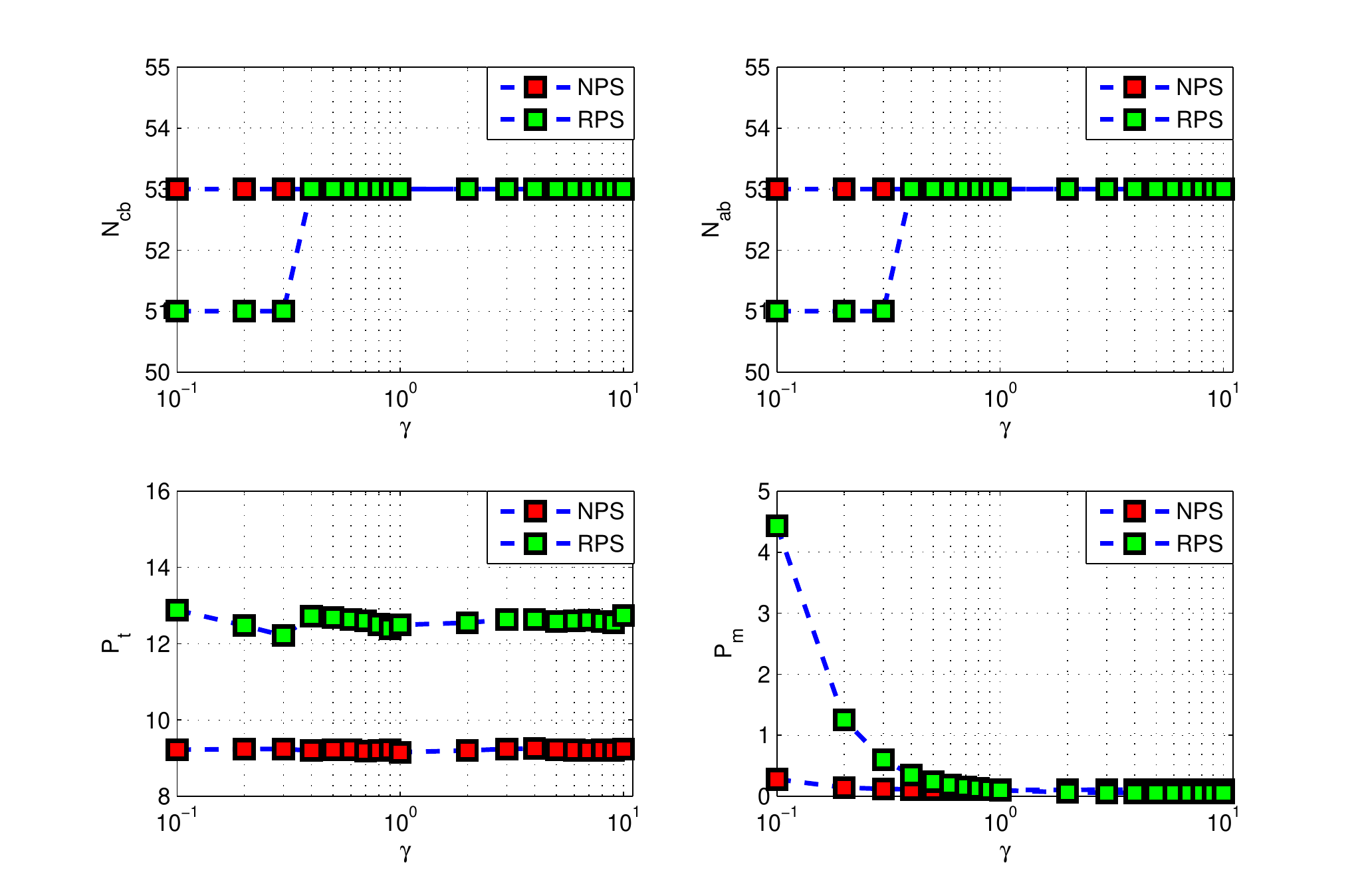}}\centering
\caption{\label{weig} The effect of weight proportion $\gamma$ on the final configuration of power systems.$N_{cb}$ and $N_{ab}$ refer to the number of connected branches and the number of active branches with transmission power at the end of cascading failures, respectively. $P_t$ and $P_m$ represent the total transmission power and the total amount of load shedding with respect to the original power, respectively.}
\end{figure}
The parameter $W$ describes the bus weight in the optimal load shedding. To be specific, the buses with smaller weights are deemed less significant in power networks, thus WAPCS prefers to shed load on these buses during emergency. This section aims
to investigate the effects of $W$ on the performance of optimal load shedding. For simplicity, we assign the same weight $W_l$ to load buses and the same weight $W_g$ to generator buses and take into account the weight distribution and proportion $\gamma=W_g/W_l$ on the two types of buses. Numerical simulations are conducted to implement NPS and RPS on the IEEE 57 Bus Systems with the same initial disturbance, respectively ($i.e.$, severing Branch $10$). For NPS, the load shedding is implemented at the $4$th cascading step, while remedial actions are taken at the $3$rd and $4$th cascading steps for RPS. Then we look into the power transmission and intact branches in the final cascading step by tuning the parameter $\gamma$ from $0.1$ to $10$ gradually. Figure \ref{weig} presents the dependence of connected branches $N_{cb}$, active branches $N_{ab}$, transmission power $P_t$ and load shedding $P_m$ on the parameter $\gamma$ in the final configuration. It is observed that NPS is insensitive to the variation of $\gamma$, and all the four measures keep stable. In contrast, RPS behaves differently as $\gamma$ varies in the range of $[0.1,1]$. Specifically, $N_{cb}$ and $N_{ab}$ jump from $51$ to $53$ as $\gamma$ increases from $0.3$ to $0.4$. In addition, the amount of load shedding $P_m$ declines greatly as $\gamma$ varies from $0.1$ to $0.3$. It is worth pointing out that RPS outperforms NPS in terms of total power transmission after terminating the cascading outages. Compared with NPS, RPS ensures the more transmission power with the fewer amount of load shedding when $\gamma$ is larger than $1$. In both NPS and RPS, we can observe that $N_{cb}$ is always equal to $N_{ab}$ with the same parameter $\gamma$.

\subsection{NPS vs RPS}
Obviously, NPS is superior to RPS in terms of computation complexity, and it takes the less time for NPS to work out Optimization Problem (\ref{convex}) according to Saddle Point Dynamics (\ref{saddle}). This enables NPS to implement
the optimal load shedding at the earlier stage of cascading blackouts. According to simulation results, RPS ensures more transmission power in power networks by taking consecutive remedial actions during emergency. Considering that more control variables are available in RPS, it is able to provide more flexible solutions of load shedding to control and protect power systems against cascading outages. In theory, RPS is able to achieve
the termination of cascading outages by shedding relatively fewer loads on buses. Thus, RPS can be applied to the regional protection of power networks for more efficient power transmission, while NPS is more suitable for the wide-area protection and control of power systems with limited response time.

\section{Conclusions}\label{sec:con}
In this paper, we developed a protection architecture of load shedding to prevent the cascading failure of power systems in the framework of convex optimization. Moreover, two types of protection schemes were introduced to interdict the cascading process of power systems. The first scheme takes remedial measures based on prediction of cascading process at a given cascading step, while the second one is coupled with the cascading model of transmission lines by shedding load at consecutive cascading steps. We proved that these two schemes are able to achieve the optimal load shedding in theory. Numerical validation on IEEE 57 Bus Systems has demonstrated the effectiveness of the proposed protection schemes. Future work may include the robust protection strategies for power systems with uncertain parameters, the consideration of high-voltage direct current (HVDC) links and the flexible alternating current transmission system (FACTS) in power transmission and distributed load shedding to enhance the resilience of power systems.

\section*{Acknowledgment}
This work is partially supported by the Future Resilient Systems Project at the Singapore-ETH Centre (SEC), which is funded by the National Research Foundation of Singapore (NRF) under its Campus for Research Excellence and Technological Enterprise (CREATE) program. It is also partially supported by Ministry of Education of Singapore under contract MOE2016-T2-1-119.

\section*{Appendix A: Proof of Proposition \ref{iff} }

Since the objective function and constraint conditions are convex, we have
\begin{equation*}
\begin{split}
J(P,W)&\geq J(P^*,W)+\nabla J(P^*,W)^T(P-P^*)\\
(P^m_{ij})^2-\sigma_{ij}^2&\geq \left[(P^{m*}_{ij})^2-\sigma_{ij}^2\right]+\nabla\left[(P^{m*}_{ij})^2-\sigma_{ij}^2\right]^T(P-P^*)\\
P_i-\bar{P}_i&\geq P^*_i-\bar{P}_i+e_i^T(P-P^*)\\
\underline{P}_i-P_i&\geq \underline{P}_i-P^*_i+e_i^T(P-P^*)
\end{split}
\end{equation*}
Multiplying the above inequalities by the optimal Lagrangian multipliers and summing them up yields
\begin{equation*}
\begin{split}
&~~~~L(P,\lambda^*,\tau^*)\\
&=J(P,W)+\sum_{(i,j)\in \mathcal{E}}\lambda^*_{ij}\left[(P^m_{ij})^2-\sigma_{ij}^2\right]
+\sum_{i=1}^{n_b}\bar{\tau}^*_i(P_i-\bar{P}_i)+\sum_{i=1}^{n_b}\underline{\tau}^*_i(\underline{P}_i-P_i)\\
&\geq J(P^*,W)+\sum_{(i,j)\in \mathcal{E}}\lambda^*_{ij}\left[(P^{m*}_{ij})^2-\sigma_{ij}^2\right]
+\sum_{i=1}^{n_b}\bar{\tau}^*_i(P^*_i-\bar{P}_i)+\sum_{i=1}^{n_b}\underline{\tau}^*_i(\underline{P}_i-P^*_i)\\
&+\left\{\nabla J(P^*,W)+\sum_{(i,j)\in \mathcal{E}}\lambda^*_{ij}\nabla\left[(P^{m*}_{ij})^2-\sigma_{ij}^2\right]+\sum_{i=1}^{n_b}(\bar{\tau}^*_i-\underline{\tau}^*_i)e_i\right\}^T(P-P^*)\\
&=L(P^*,\lambda^*,\tau^*)\\
&+\left\{\nabla J(P^*,W)+\sum_{(i,j)\in \mathcal{E}}\lambda^*_{ij}\nabla\left[(P^{m*}_{ij})^2-\sigma_{ij}^2\right]+\sum_{i=1}^{n_b}(\bar{\tau}^*_i-\underline{\tau}^*_i)e_i\right\}^T(P-P^*)\\
\end{split}
\end{equation*}
It follows from the KKT gradient condition that
$$
\nabla J(P^*,W)+\sum_{(i,j)\in \mathcal{E}}\lambda^*_{ij}\nabla\left[(P^{m*}_{ij})^2-\sigma_{ij}^2\right]+\sum_{i=1}^{n_b}(\bar{\tau}^*_i-\underline{\tau}^*_i)e_i=\mathbf{0}
$$
which implies
$$
L(P,\lambda^*,\tau^*)\geq L(P^*,\lambda^*,\tau^*)
$$
Moreover, the KKT conditions allow to establish the equation
$$
\sum_{(i,j)\in \mathcal{E}}\lambda^*_{ij}\left[(P^{m*}_{ij})^2-\sigma_{ij}^2\right]
+\sum_{i=1}^{n_b}\bar{\tau}^*_i(P^*_i-\bar{P}_i)+\sum_{i=1}^{n_b}\underline{\tau}^*_i(\underline{P}_i-P^*_i)=0
$$
which leads to
\begin{equation*}
\begin{split}
&~~~~L(P^*,\lambda,\tau)\\
&=J(P^*,W)+\sum_{(i,j)\in \mathcal{E}}\lambda_{ij}\left[(P^{m*}_{ij})^2-\sigma_{ij}^2\right]
+\sum_{i=1}^{n_b}\bar{\tau}_i(P^*_i-\bar{P}_i)+\sum_{i=1}^{n_b}\underline{\tau}_i(\underline{P}_i-P^*_i)\\
&\leq J(P^*,W)+\sum_{(i,j)\in \mathcal{E}}\lambda^*_{ij}\left[(P^{m*}_{ij})^2-\sigma_{ij}^2\right]
+\sum_{i=1}^{n_b}\bar{\tau}^*_i(P^*_i-\bar{P}_i)+\sum_{i=1}^{n_b}\underline{\tau}^*_i(\underline{P}_i-P^*_i) \\
&=L(P^*,\lambda^*,\tau^*)
\end{split}
\end{equation*}
Therefore, we get
$$
L(P^*,\lambda,\tau)\leq L(P^*,\lambda^*,\tau^*)\leq L(P,\lambda^*,\tau^*),\quad \forall~ \lambda_{ij}\geq0,~\bar{\tau}_i\geq0,~\underline{\tau}_i\geq0
$$
which indicates that $(P^*,\lambda^*,\tau^*)$ is a saddle point of the  Lagrangian function $L(P,\lambda,\tau)$.

Next, we demonstrate that a saddle point of the  Lagrangian function $L(P,\lambda,\tau)$ implies the optimality with the KKT conditions. Considering that
$$
L(P^*,\lambda^*,\tau^*)\leq L(P,\lambda^*,\tau^*), \quad \forall~\underline{P}_i \leq P_i\leq \bar{P}_i,
$$
$P^*$ solves the optimization problem $\min_{P}L(P,\lambda^*,\tau^*)$ and $P^*$ is a stationary point of $L(P,\lambda^*,\tau^*)$.
It thus follows
$$
\nabla_P L(P^*,\lambda^*,\tau^*)=\mathbf{0},
$$
which is equivalent to the KKT gradient condition. Moreover, it follows from
$$
L(P^*,\lambda,\tau)\leq L(P^*,\lambda^*,\tau^*),\quad \forall~ \lambda_{ij}\geq0,~\bar{\tau}_i\geq0,~\underline{\tau}_i\geq0
$$
that
$$
(P^{m*}_{ij})^2-\sigma_{ij}^2\leq0, \quad P^*_i-\bar{P}_i\leq0, \quad \underline{P}_i-P^*_i\leq0
$$
and
$$
J(P^*,W)=L(P^*,\mathbf{0},\mathbf{0})\leq L(P^*,\lambda^*,\tau^*)
$$
which leads to
$$
\lambda^*_{ij}\left[(P^{m*}_{ij})^2-\sigma_{ij}^2\right]=0,\quad \bar{\tau}^*_i(P^*_i-\bar{P}_i)=0,\quad \underline{\tau}^*_i(\underline{P}_i-P^*_i)=0
$$
Thus we get the KKT condition of complementary slackness. Finally, we show the optimality of $P^*$ to minimize the cost function. The KKT condition of complementary slackness and definition of saddle point enable us to obtain
$$
J(P^*,W)=L(P^*,\lambda^*,\tau^*)\leq L(P,\lambda^*,\tau^*),\quad \underline{P}_i\leq P_i\leq \bar{P}_i
$$
Since Slater's condition holds for convex optimization problem (\ref{convex}), we have
$$
J(P^*,W)\leq \inf_{\quad \underline{P}_i\leq P_i\leq \bar{P}_i}L(P,\lambda^*,\tau^*)=\theta(\lambda^*,\tau^*)\leq \max_{\lambda,\tau}\theta(\lambda,\tau)=\min_{P}J(P,W)
$$
which implies
$$
J(P^*,W)=\min_{P}J(P,W)
$$
Therefore, $P^*$ is the optimal solution to Problem (\ref{convex}). The proof is thus completed.

\section*{Appendix B: Proof of Proposition \ref{converge} }
Design Lyapunov function as follows
$$
V(P,\lambda,\tau)=\frac{1}{2}(\|P-P^*\|^2+\|\lambda-\lambda^*\|^2+\|\tau-\tau^*\|^2)
$$
The time derivative of $V(P,\lambda,\tau)$ along Saddle Point Dynamics (\ref{saddle}) gives
\begin{equation*}
\begin{split}
\dot{V}(P,\lambda,\tau)&=-(P-P^*)^T\nabla_P~L(P,\lambda,\tau)+(\lambda-\lambda^*)^T[\nabla_{\lambda}~L(P,\lambda,\tau)]^+_{\lambda}\\
                       &~~~+(\tau-\tau^*)^T[\nabla_{\tau}~L(P,\lambda,\tau)]^+_{\tau}
\end{split}
\end{equation*}
For each element in $\lambda$ and $\tau$, Eq.~(\ref{switch}) leads to
$$
(\lambda_{ij}-\lambda_{ij}^*)^T[\nabla_{\lambda_{ij}}~L(P,\lambda,\tau)]^+_{\lambda_{ij}}\leq(\lambda_{ij}-\lambda_{ij}^*)^T\nabla_{\lambda_{ij}}~L(P,\lambda,\tau)
$$
and
$$
(\tau_i-\tau_i^*)^T[\nabla_{\tau_i}~L(P,\lambda,\tau)]^+_{\tau_i}\leq(\tau_i-\tau_i^*)^T\nabla_{\tau_i}~L(P,\lambda,\tau)
$$
Therefore, one has
\begin{equation*}
\begin{split}
\dot{V}(P,\lambda,\tau)&\leq-(P-P^*)^T\nabla_P~L(P,\lambda,\tau)+(\lambda-\lambda^*)^T\nabla_{\lambda}~L(P,\lambda,\tau)\\
                       &~~~+(\tau-\tau^*)^T\nabla_{\tau}~L(P,\lambda,\tau)
\end{split}
\end{equation*}
Since $L(P,\lambda,\tau)$ is convex in $P$ and concave in $\lambda$ and $\tau$, we have
$$
-(P-P^*)^T\nabla_P~L(P,\lambda,\tau)\leq L(P^*,\lambda,\tau)-L(P,\lambda,\tau)
$$
and
$$
(\lambda-\lambda^*)^T\nabla_{\lambda}~L(P,\lambda,\tau)+(\tau-\tau^*)^T\nabla_{\tau}~L(P,\lambda,\tau)\leq L(P,\lambda,\tau)-L(P,\lambda^*,\tau^*)
$$
which leads to
\begin{equation*}
\begin{split}
\dot{V}(P,\lambda,\tau)&\leq L(P^*,\lambda,\tau)-L(P,\lambda,\tau)+L(P,\lambda,\tau)-L(P,\lambda^*,\tau^*)\\
                       &=L(P^*,\lambda,\tau)-L(P,\lambda^*,\tau^*)\\
                       &=\left[L(P^*,\lambda,\tau)-L(P^*,\lambda^*,\tau^*)\right]+\left[L(P^*,\lambda^*,\tau^*)-L(P,\lambda^*,\tau^*)\right]
\end{split}
\end{equation*}
Moreover, it follows from
$$
L(P^*,\lambda,\tau)-L(P^*,\lambda^*,\tau^*)\leq0
$$
and
$$
L(P^*,\lambda^*,\tau^*)-L(P,\lambda^*,\tau^*)<0
$$
that $\dot{V}(P,\lambda,\tau)<0$. This indicates $V(P,\lambda,\tau)$ converges to $0$ (and $P$ also converges to the optimal value $P^*$) as time goes to infinity. The proof is thus completed.

\section*{Appendix C: Proof of Proposition \ref{cp_con}}

Design Lyapunov function as follows
$$
V(P^m,P^{m-1},\lambda,\tau)=\frac{1}{2}(\sum_{k=m-1}^{m}\|P^k-P^{k*}\|^2+\sum_{k=m-1}^{m}\|\tau^k-\tau^{k*}\|^2+\|\lambda-\lambda^*\|^2)
$$
The time derivative of $V(P,\lambda,\tau)$ along Saddle Point Dynamics (\ref{saddlecas}) gives
\begin{equation*}
\begin{split}
\dot{V}(P^m,P^{m-1},\lambda,\tau)&=\sum_{k=m-1}^{m}(P^{k*}-P^k)^T\nabla_{P^k}~L(P^m,P^{m-1},\lambda,\tau)\\
                                 &+\sum_{k=m-1}^{m}(\tau^k-\tau^{k*})^T[\nabla_{\tau}~L(P^m,P^{m-1},\lambda,\tau)]^+_{\tau^k}\\
                                 &+(\lambda-\lambda^*)^T[\nabla_{\lambda}~L(P^m,P^{m-1},\lambda,\tau)]^+_{\lambda}\\
\end{split}
\end{equation*}
Following the same inference in Proposition \ref{converge}, one obtains
\begin{equation*}
\begin{split}
\dot{V}(P^m,P^{m-1},\lambda,\tau)&\leq\sum_{k=m-1}^{m}(P^{k*}-P^k)^T\nabla_{P^k}~L(P^m,P^{m-1},\lambda,\tau)\\
                                 &+\sum_{k=m-1}^{m}(\tau^k-\tau^{k*})^T\nabla_{\tau}~L(P^m,P^{m-1},\lambda,\tau)\\
                                 &+(\lambda-\lambda^*)^T\nabla_{\lambda}~L(P^m,P^{m-1},\lambda,\tau)\\
\end{split}
\end{equation*}
Since $L(P^m,P^{m-1},\lambda,\tau)$ is strictly convex in $P^m$ and $P^{m-1}$ and concave in $\lambda$ and $\tau$, one has
$$
\sum_{k=m-1}^{m}(P^{k*}-P^k)^T\nabla_{P^k}~L(P^m,P^{m-1},\lambda,\tau)\leq L(P^{m*},P^{m-1*},\lambda,\tau)-L(P^m,P^{m-1},\lambda,\tau)<0
$$
and
\begin{equation*}
\begin{split}
&~~~~\sum_{k=m-1}^{m}(\tau^k-\tau^{k*})^T\nabla_{\tau}~L(P^m,P^{m-1},\lambda,\tau)+(\lambda-\lambda^*)^T\nabla_{\lambda}~L(P^m,P^{m-1},\lambda,\tau)\\
&\leq L(P^{m},P^{m-1},\lambda,\tau)-L(P^m,P^{m-1},\lambda^*,\tau^*)\leq 0
\end{split}
\end{equation*}
which implies
\begin{equation*}
\begin{split}
\dot{V}(P^m,P^{m-1},\lambda,\tau)&\leq L(P^{m*},P^{m-1*},\lambda,\tau)-L(P^m,P^{m-1},\lambda,\tau)\\
                                 &+L(P^{m},P^{m-1},\lambda,\tau)-L(P^m,P^{m-1},\lambda^*,\tau^*)\\
                                 &<0
\end{split}
\end{equation*}
Therefore, $V(P^m,P^{m-1},\lambda,\tau)$ converges to $0$ as time goes to infinity, which indicates the conclusion of this proposition.  Thus we complete the proof.

\end{document}